\documentclass{amsart}
\newtheorem{theorem}{Theorem}[section]

\newtheorem{proposition}[theorem]{Proposition}
\newtheorem{lemma}[theorem]{Lemma}
\newtheorem{conjecture}[theorem]{Conjecture}

\usepackage{hyperref,amssymb}
\begin{document}
	
	\title[Irredundant bases of primitive groups]{Cardinalities of irredundant bases of finite primitive groups}

	\author[F. Mastrogiacomo]{Fabio Mastrogiacomo}
	\address{Fabio Mastrogiacomo, Dipartimento di Matematica ``Felice Casorati", University of Pavia, Via Ferrata 5, 27100 Pavia, Italy} 
	\email{fabio.mastrogiacomo01@universitadipavia.it}
	\maketitle
	\begin{abstract}        Let $G$ be a finite permutation group acting on a set $\Omega$. An ordered sequence 
		$(\omega_1,\ldots,\omega_\ell)$ of elements of $\Omega$ is an irredundant base for $G$ if the pointwise stabilizer of the sequence is trivial and no point is fixed by the stabilizer of its predecessors. We show that any interval of natural numbers can be realized as the set of cardinalities of irredundant bases for some finite primitive group.
	\end{abstract}

	\section{Introduction}
	Let $G$ be a permutation group on $\Omega$. 
	Given an ordered sequence $(\omega_1,\ldots,\omega_\ell)$ of elements of $\Omega$, we study the associated stabilizer chain:
	$$G\ge G_{\omega_1}\ge G_{\omega_1,\omega_2}\ge \cdots \ge G_{\omega_1,\omega_2,\ldots,\omega_\ell}.$$
	If all the inclusions given above are strict, then the stabilizer chain is called \textit{\textbf{irredundant}}. Furthermore, if the group $G_{\omega_1,\ldots,\omega_\ell}$ is the identity, then the sequence $(\omega_1,\ldots,\omega_\ell)$ is called an \textit{\textbf{irredundant base}}. The size of the smallest possible base is denoted by $b(G,\Omega)$, while the size of the longest possible irredundant base is denoted by $I(G,\Omega)$. The term \emph{irredundant} here is used to emphasize the distinction from another type of bases, the minimal bases (see \cite{DMS} for more details). 
	
	In this paper, given a finite permutation group $G$ with domain $\Omega$, we are interested in the following subset of integers associated to $G$ arising from irredundant bases. We let
	\begin{align*}
		\mathcal{I}(G,\Omega):=\{\ell\in\mathbb{N}\mid &\exists \omega_1,\ldots,\omega_\ell\in \Omega\hbox{ such that }\\&(\omega_1,\ldots,\omega_\ell)\ \hbox{ is an irredundant base}\}.	
	\end{align*}
	The first result about this set is due to Cameron, who has proved that for any permutation group $G$ acting on $\Omega$, the set $\mathcal{I}(G,\Omega)$ is an interval of natural numbers.
	In \cite{DMS}, the authors investigate which interval 
	can appear as $\mathcal{I}(G,\Omega)$ for some finite permutation group $G$. They further inquire whether this group can be taken to be transitive or primitive. Concerning the first case, \cite[Theorem 1.3]{DMS} provides a complete answer. Specifically, it is shown that for any interval $X$ of natural numbers (not containing the number 1)\footnote{If $1 \in X$ and if $G$ is transitive, then $G$ is regular, so that $\mathcal{I}(G,\Omega) = \{1\}$. }, there exists a finite transitive permutation group $G$ acting on $\Omega$ such that
	\[
		\mathcal{I}(G,\Omega)= X.
	\]
	The construction is explicit and involves the product action of finitely many symmetric groups in their natural actions. Thus, the obtained group is transitive but imprimitive.\\ For what regards the primitive case, they have shown that there exists an interval which is never realized by a symmetric group $\mathrm{Sym}(n)$ in its natural action on the collection of $k$-subsets, with $k \leq n/2$. Motivated by this, they have proposed the following conjecture.
	\begin{conjecture}\cite[Conjecture $5.1$]{DMS}\label{conj1}
		There exists an interval  $X$ of positive integers, not containing $1$, such that no primitive group $G$ acting on some $\Omega$ satisfies 
		\[
			\mathcal{I}(G,\Omega)=X.
		\]
	\end{conjecture}

	Here, we disprove this conjecture. Our main result is the following.
	\begin{theorem}\label{mainthm}
		Let $X$ be an interval of natural numbers, not containing $1$. Then, there exists a primitive permutation group $G$ acting on $\Omega$ such that
		\[
			X = \mathcal{I}(G,\Omega).
		\]
	\end{theorem}
	As in the transitive case, the construction is explicit and quite uniform. Indeed, we only need to distinguish whether the smallest element of the interval $X$ is $2$ or greater than $2$. \\
	The proof, based on the same idea of \cite[Example $5.1$]{GL}, is in the subsequent section, and is organized as follows.\\ Firstly, we recall some basic properties of the Suzuki groups, and we compute the interval $\mathcal{I}$ for the automorphism group of $\,^2B_2(q)$ in its action on the collection of $2$-subsets of the Suzuki ovoid. This will be used for the case where the smallest element of $X$ is $2$. For the remaining cases, we compute the interval $\mathcal{I}$ for the automorphism group of the general affine group in its natural action on its underlying vector space. \\
	
	\noindent {\bf Acknowledgements.} {\sl The author is grateful to Pablo Spiga for many helpful discussions.\\
	\noindent The author is a member of the GNSAGA INdAM research group and kindly acknowledges their support.}
		
	\section{Proof of Theorem~\ref{mainthm}}
	We start by making some basic constructions that will help us reach our conclusion.
	
	Take $G_0 = \,^2B_2(q)$, with $q=2^{2m+1}$, and $G = G_0 \rtimes \mathrm{Aut}(\mathbb{F}_q)$. Set $f=2m+1$. We first recall the definition of the group $G_0$ (following \cite{dixon_mortimer}) with some basic properties of $G_0$ and $G$. Define
	\[
		\Delta = \{(\eta_1,\eta_2,\eta_3) \in \mathbb{F}_q^3 \, | \, \eta_3 = \eta_1 \eta_2 + \eta_1^{\sigma+2}+\eta_2^{\sigma}\} \cup \{\infty\},
	\]
	where $\sigma \in \mathrm{Aut}(\mathbb{F}_q)$ is the automorphism defined by $\xi \mapsto \xi^{2^{m+1}}$. This set is usually referred to as the \emph{Suzuki ovoid}, and has cardinality $|\Delta| = q^2+1$.\\
	For $\alpha,\beta,\gamma \in \mathbb{F}_q$, with $\gamma \neq 0$, define the following permutations of $\Delta$ fixing $\infty$
	\begin{align*}
			t_{\alpha,\beta}:(\eta_1,\eta_2,\eta_3) \mapsto& (\eta_1+\alpha,\eta_2+\beta+\beta^\sigma \eta_1, \\
			&\eta_3+\alpha\beta + \alpha^{\sigma+2}+\beta^\sigma+\alpha\eta_2+\alpha^{\sigma+1}\eta_1+\beta\eta_1), \\
			n_\gamma: (\eta_1,\eta_2,\eta_3) \mapsto&(\gamma \eta_1, \gamma^{\sigma+1}\eta_2,\gamma^{\sigma+2}\eta_3).
	\end{align*}
	Next, define the involution $w$ by
	\begin{align*}
		w : (\eta_1,\eta_2,\eta_3) &\leftrightarrow \left(\frac{\eta_2}{\eta_3},\frac{\eta_1}{\eta_3},\frac{1}{\eta_3}\right)\, \text{ for } \eta_3 \neq 0, \\
		\infty &\leftrightarrow (0,0,0).
	\end{align*}
	We now define the group $G_0 = \,^2B_2(q)$ as the group generated by $w$ and all the permutations $t_{\alpha,\beta}$ and $n_\gamma$.
	In \cite{Suz}, it is proved that $G_0$ acts doubly transitive on $\Delta$. Moreover, from \cite[Theorem 3]{BLS} we see that $b(G,\Delta) = 3$.\\
	Consider now $$\Omega = \{ \omega \subseteq \Delta \, : \, |\omega| = 2\}.$$ Since $G_0$ is doubly transitive on $\Delta$, $G_0$ acts transitively on $\Omega$. Actually, it turns out that this action is primitive. Indeed, the stabilizer of a point of $\Omega$ is the dihedral group of order $2(q+1)$, and it is shown in \cite{Suz} that this is a maximal subgroup of $G_0$. Finally, we can extend both the actions of $G_0$ on $\Delta$ and $\Omega$ to actions of $G$ on these sets, and still obtain primitive actions. 

	\begin{lemma}\label{suzzu0}
		The base size of $G$ in its action on $\Omega$ is $2$, that is $b(G,\Omega) = 2$.
	\end{lemma}
	\begin{proof}
		Take $\omega_1 = \{(0,0,0),\infty\}$. Since $b(G,\Delta) = 3$, there exists $ \delta= (\eta_1,\eta_2,\eta_3) \in \Delta$ such that $G_{(0,0,0),\infty,\delta} = 1$ (we are using the $2$-transitivity of $G$ on $\Delta$ here). Take $\omega_2 = \{\infty, \delta\}$. Thus, $G_{\omega_1,\omega_2} = G_{(0,0,0),\infty,\delta} = 1$.
	\end{proof}
	Observe that, since $b(G_0,\Omega) \leq b(G,\Omega)$, we also have $b(G_0,\Omega)=2$.\\
	The following lemma is a basic observation made in \cite{LS}.
	\begin{lemma}
		Let $G$ be a primitive almost simple group with $b(G,\Omega) = 2$. Then $G$ admits an irredundant base of cardinality greater than $2$.
	\end{lemma}
	In particular, $G$, in its action on $\Omega$, has an irredundant base of cardinality $3$. The next lemma shows that this is the maximum possible size.
	\begin{lemma}
		The size of the longest possible irredundant base for $G_0$ in its action on $\Omega$ is $3$, that is $I(G_0,\Omega) = 3$.
	\end{lemma}
	\begin{proof}
		Set $\omega_1 = \{(0,0,0),\infty\}$, so that $(G_0)_{\omega_1} = D_{2(q-1)}$. Since $G_0$ admits an irredundant base of cardinality $3$, there exists $\omega_2 \in \Omega$ such that $G_{\omega_1} > G_{\omega_1,\omega_2} \neq 1$. Now define $$H = \{ g \in (G_0)_{(0,0,0),\infty}\, : \, \omega_2^g = \omega_2\}.$$
		We have $[(G_0)_{\omega_1,\omega_2}:H]\leq2$, since an element in $G_{\omega_1,\omega_2}$ can either fix $(0,0,0)$ and $\infty$ or swap them.\\
		We now claim that $H=1$. For, take $g \in H$. In particular, $$g \in (G_0)_{(0,0,0),\infty} = \langle n_\gamma, \, \gamma \in \mathbb{F}_q \setminus \{0\} \rangle,$$ which is a cyclic group. Thus, we can assume that $g = n_\gamma$, for some $\gamma \in \mathbb{F}\setminus \{0\}$. Set now $\omega_2 = \{(\alpha_1,\alpha_2,\alpha_3),(\beta_1,\beta_2,\beta_3)\}$. If $g$ fixes both points of $\omega_2$, then, for example, $$(\alpha_1,\alpha_2,\alpha_3) = (\alpha_1,\alpha_2,\alpha_3)n_\gamma = (\gamma \alpha_1, \gamma^{\sigma+1}\alpha_2, \gamma^{\sigma+2}\alpha_3).$$ By comparing the first coordinates we see $\gamma=1$, and so $g=1$. Suppose now that $g$ interchanges $(\alpha_1,\alpha_2,\alpha_3)$ and $(\beta_1,\beta_2,\beta_3)$. In particular, again by looking at the first coordinates, we have $\gamma\alpha_1= \beta_1$ and $\gamma \beta_1 = \alpha_1$, and so $\gamma^2 = 1$. Since we are working in characteristic $2$, this means $\gamma=1$.\\
		In conclusion, $|(G_0)_{\omega_1,\omega_2}|\leq2$ for every $\omega_2 \in \Omega$, and hence $G_0$ has an irredundant base of length at most $3$.
	\end{proof}
	With this information in hand, the next step is to compute $I(G,\Omega)$. The following lemma accomplishes this task.\\ Here, for a natural number $n$, $\pi(n)$ denotes the number of prime divisors of $n$, counted with multiplicity.
	\begin{lemma} \label{suzzu1}
		The size of the longest possible irredundant base for $G$ in its action on $\Omega$ is $3+\pi(f)$, that is $I(G,\Omega) = 3+ \pi(f)$.
	\end{lemma}
	\begin{proof}
		We first establish an irredundant base of length $3$ for $G_0$. Take $$\omega_1 = \{(0,0,0),\infty\}, \, \omega_2= \{(1,0,1),(0,1,1)\}, \, \omega_3= \{(0,0,0),(1,1,1)\}.$$ Then we have
		\begin{align*}
			(G_0)_{\omega_1} = \langle n_\gamma, w, \, : \, \gamma \in \mathbb{F}_q \setminus\{0\}\rangle> 
			(G_0)_{\omega_1,\omega_2} = \{1,w\}>
			(G_0)_{\omega_1,\omega_2,\omega_3} = 1.
		\end{align*}
		Thus, $(\omega_1,\omega_2,\omega_3)$ is an irredundant base for $G_0$. Moreover, $$G_{\omega_1,\omega_2,\omega_3} = \mathrm{Aut}(\mathbb{F}_q).$$ 
		Suppose now that the prime factorization of $f$ is $f = p_1p_2\cdots p_r$, where $r = \pi(f)$. Take $\zeta_i$ to be a primitive element of $\mathbb{F}_{2^{p_1p_2\cdots p_i}}$, for $i=1,2,\dots,r$. Then, we can extend the sequence $(\omega_1,\omega_2,\omega_3)$ by successively stabilizing the following elements
		\[
			\{(\zeta_1,0,\zeta_1^{\sigma+2}),\infty\},	\{(\zeta_2,0,\zeta_2^{\sigma+2}),\infty\}, \dots,
			\{(\zeta_r,0,\zeta_r^{\sigma+2}),\infty\}.
		\]
		Since $|\mathrm{Aut}(\mathbb{F}_q)|= p_1p_2\cdots p_r$, this is the longest possible irredundant base for $G$.
	\end{proof}
	In conclusion, we have the following proposition.
	\begin{proposition}\label{propSuz}
		We have that 
		\[
			\mathcal{I}(G,\Omega) = \{2,3,\dots,3+\pi(f)\}.
		\]
	\end{proposition}

	Moving on, take $G_0 = \mathrm{AGL}_d(q)$ and $G = G_0 \rtimes \mathrm{Aut}(\mathbb{F}_q)$, with $q=p^f$. We consider the natural action of $G$ (and $G_0$) on the vectors of $\mathbb{F}_q^d$. It is trivial to see that every irredundant base of $G_0$ has cardinality $d+1$. This is because the stabilizer of the zero vector is $\mathrm{GL}_d(q)$, and we recover the action of the general linear group on the non-zero vectors. Now an irredundant base for this action is simply a basis for the vector space.
	\begin{proposition}\label{affine}
		We have that 
		\[
			\mathcal{I}(G,\mathbb{F}_q^d) = \{d+1,d+2,\dots,d+1+\pi(f)\}.
		\]
	\end{proposition}
	\begin{proof}
		Consider the sequence $(0,e_1,\dots,e_d)$. Since this is a base for $G_0$ fixed by $\mathrm{Aut}(\mathbb{F}_q)$, the stabilizer in $G$ of this sequence is $\mathrm{Aut}(\mathbb{F}_q)$. As before, write the prime factorization of $f$ as $f = p_1p_2\cdots p_r$, where $r = \pi(f)$, and take $\zeta_i$ to be a primitive element of $\mathbb{F}_{2^{p_1p_2\cdots p_i}}$, for $i=1,2,\dots,r$. Thus, we obtain a base for $G$ by successively stabilizing the elements
		\[
			\zeta_1 e_1, \zeta_2 e_1, \dots, \zeta_r e_1.
		\]
		Since $G_0$ admits irredundant bases only of cardinality $d+1$, and since $|\mathrm{Aut}(\mathbb{F}_q)| = p_1p_1\cdots p_r$, this is the longest possible irredundant base for $G$.
		\\
		We now show that $b(G,\mathbb{F}_q^d) = d+1$. Since $G_0 \leq G$ and $d+1 = b(G_0,\mathbb{F}_q^d) \leq b(G,\mathbb{F}_q^d)$, it is sufficient to establish an irredundant base for $G$ of size $d+1$.	\\
		Let $\mu \in \mathbb{F}_q$ be a generator of the multiplicative group of the field. Consider the sequence
		\[
			(0,e_1,e_2,\dots,e_{d-1},\mu e_{d-1}+e_d).
		\]
		We claim that this is an irredundant base. Take $g \in G_{0,e_1,\dots,e_{d-1}}$. Then, 
		\[
			(\mu e_{d-1} + e_d)g = \mu^\varphi e_{d-1}+ \lambda e_d,
		\]
		for some $\lambda \in \mathbb{F}_q$ and $\varphi \in \mathrm{Aut}(\mathbb{F}_q)$. Therefore, $g$ fixes $\mu e_{d-1}+e_d$ if and only if $$\mu^\varphi e_{d-1}+ \lambda e_d = \mu e_{d-1}+e_d,$$ and this implies that $\lambda = 1$ and $\varphi = \mathrm{Id}$.
	\end{proof}
	We are now ready to prove our main result.
	\begin{proof}[Proof of Theorem \ref*{mainthm}]
	Let $X = \{a,a+1,\dots,b\}$, with $a>1$. If $b=a$, take $G = \mathrm{Sym}(a+1)$ in its natural action. Suppose then that $b>a$. \\Firstly, suppose that $a=2$ and define $G_0 = \,^2B_2(q)$ and $G =  G_0 \rtimes \mathrm{Aut}(\mathbb{F}_q)$, with $q=2^f$, and take $\Omega$ as before (so that $\Omega$ is the collection of $2$-subsets of the Suzuki ovoid).\\
	If $b=3$, then we have seen that $\mathcal{I}(G_0,\Omega)=\{2,3\}$. \\If $b>3$, take $f$ to be the product of $b-3$ distinct primes. Then, Proposition~\ref{propSuz} shows that
	\[
	\mathcal{I}(G,\Omega) = \{2,\dots,3+\pi(f)\} = \{2,\dots,b\}.
	\]
	Assume now then that $b>a\geq3$. Take $q=p^f$ to be the product of $b-a$ distinct primes, for some prime $p$. Define $G = \mathrm{AGL}_{a-1}(q) \rtimes \mathrm{Aut}(\mathbb{F}_q)$ in its action on $\mathbb{F}_q^{a-1}$. Then, Proposition~\ref{affine} shows that 
	\[
		\mathcal{I}(G,\mathbb{F}_q^{a-1}) = \{a,\dots,a+\pi(f)\} = \{a,\dots,b\}.
	\]
	This concludes the proof.
	\end{proof}

\thebibliography{20}
	\bibitem{BLS}T.~C.~Burness, M.~W.~Liebeck, A.~Shalev, Base sizes for simple groups and a conjecture of Cameron, \textit{Proceedings of the London Mathematical Society}, Volume 98, Issue 1, January 2009, Pages 116–162
	\bibitem{DMS}F.~Dalla Volta, F.~Mastrogiacomo, P.~Spiga, On the cardinality of irredundant and minimal  bases of finite permutation groups, \textit{J. Algebr. Comb.}.
	\bibitem{dixon_mortimer}J.~D.~Dixon,  B.~Mortimer, \textit{Permutation groups}, Graduate Texts in Mathematics \textbf{163}, Springer-Verlag, New York, 1996.
	\bibitem{GL} N.~Gill, M.~W.~Liebeck, Irredundant bases for finite groups of Lie type, \text{Pacific Journal of Mathematics} 322.2 (2023): 281-300.
	\bibitem{LS} M.~Lee, P.~Spiga, A classification of finite primitive IBIS groups with alternating socle, \textit{Journal of Group Theory} 26.5 (2023): 915-930.
	\bibitem{Suz} M.~Suzuki. On a Class of Doubly Transitive Groups, \textit{Annals of Mathematics}, vol. 75, no. 1, 1962, pp. 105–45.
	\end{document}